\documentclass[a4paper,11pt,leqno]{amsart}
\usepackage[utf8]{inputenc}

\usepackage{amsmath}
\usepackage{amsthm}
\usepackage{amsfonts}
\usepackage{amssymb}

\usepackage{tikz}
\usetikzlibrary{arrows}
\usepackage{enumerate}
\usetikzlibrary{calc}
\usepackage{colonequals}
\usepackage{enumitem}

\usepackage{verbatim}

\usepackage{mathtools}
\usepackage{esint}

\newlength{\bibitemsep}
\newlength{\bibparskip}
\let\oldthebibliography\thebibliography
\renewcommand\thebibliography[1]{%
	\oldthebibliography{#1}%
	\setlength{\parskip}{\bibitemsep}%
	\setlength{\itemsep}{\bibparskip}%
}

\newcommand{\id}{\operatorname{id}}

\newcommand{\cof}{\operatorname{cof}}

\newcommand{\loc}{\operatorname{loc}}

\newcommand{\R}{\mathbb R}

\DeclareMathOperator{\diam}{diam}

\DeclareMathOperator{\card}{card}
\DeclareMathOperator{\diver}{div}
\DeclareMathOperator{\I}{I}

\newcommand{\abs}[1]{\lvert#1\rvert}
\newcommand{\norm}[1]{\lVert#1\rVert}
\DeclareMathOperator{\esssup}{ess \, sup}

\newtheorem{theorem}{Theorem}[section]
\newtheorem{question}[theorem]{Question}
\newtheorem{conjecture}[theorem]{Conjecture}
\newtheorem*{theorem*}{Theorem}{\bf}{\it}

\newtheorem*{proposition*}{Proposition}{\bf}{\it}
\newtheorem{lemma}[theorem]{Lemma}
\newtheorem*{lemma*}{Lemma}{\bf}{\it}

\theoremstyle{definition}

\newtheorem*{definition*}{Definition}
\theoremstyle{remark}

\numberwithin{equation}{section}

\begin{document}

\title[A self-contained proof to Martio's conjecture]{A self-contained proof to Martio's conjecture in the class of BLD-maps}

\author{Ville Tengvall}
\address{Department of Mathematics and Statistics, University of Jyväskylä, P.O. Box 35 (MaD), FI-40014 University of Jyväskylä, Finland}
\email{ville.j.tengvall@jyu.fi}
\thanks{}

\date{\today}

\begin{abstract}
We provide a self-contained proof to so-called Martio's conjecture in the class of mappings of bounded length distortion. Unlike the earlier proofs,  our proof is not based on the modulus of continuity estimate of Martio from 1970.
\end{abstract}

\maketitle

\section{Introduction}\label{sec:Intro}

In this article we study the sufficient conditions for the local invertibility of mappings. Our work is motivated by the well-known \emph{inverse function theorem} which states that every continuously differentiable mapping
$$f \in C^{1}(U, \R^n) \quad \text{($U \subset \R^n$ open and connected 
	set)}$$
is a local $C^1$-diffeomorphism outside the zero set of its Jacobian determinant. We are interested in the conditions under which one can recover the local invertibility when the usual assumptions of the inverse function theorem are not satisfied. This leads us to the following two questions:
\begin{itemize}
	\item[(Q1)] How to recover local invertibility when the mappings are less than  $C^1$-regular?
	\item[(Q2)] How to recover local invertibility near the singular set of the Jacobian determinant of mappings?
\end{itemize}
We study these questions in the class of \emph{quasiregular mappings}, id est, in the class of Sobolev mappings
$$f \in W_{\loc}^{1,n}(U, \R^n) \quad \quad \text{($U \subset \R^n$ open and connected 
	set with $n \ge 2$)}$$
for which the \emph{operator norm} of the weak differential matrix satisfies the following \emph{distortion inequality}
\begin{align}\label{DistortionInequality}
\norm{Df(x)}^n \le  K \det Df(x) \colonequals K J(x; f) \colonequals KJ_f(x) \quad \text{a.e.}
\end{align}
for some constant $K \ge 1$. Homeomorphic quasiregular maps form the well-studied class of \emph{quasiconformal mappings}. For the basic properties and the background of quasiregular and quasiconformal mappings we refer to monographs \cite{Astala-Iwaniec-Martin,IwaniecMartin,Reshetnyak89,Rickman-book, VuorinenBook, VaisalaBook}. 

Next we point out that even if the definition of quasiregularity is purely analytical it still harbors a great deal of topological information per se. Especially, the distortion inequality can be applied to provide surprisingly vast amount of information on the invertibility properties of quasiregular mappings. This was first observed by Reshetnyak who originally introduced quasiregular mappings by the name of \emph{mappings of bounded distortion} and discovered their basic properties in a series of papers in 1966--1969. One of the deepest discoveries of these works was that non-constant quasiregular mappings are discrete and open, see e.g. \cite{Reshetnyak89}. This observation connected the study of quasiregular mappings to the earlier studies on \emph{branched coverings} in geometric topology, see e.g \cite{Chernavski1964, Chernavski1965, ChurchHemmingsen1960, ChurchHemmingsen1961, ChurchHemmingsen1963}.

In the critical step of the proof of Reshetnyak's celebrated discreteness and openness theorem one applies non-linear potential theory and  non-linear PDEs to transfer analytical data into topological information. This step can be carried out by studying the geometric size of the polar sets of the solutions to the quasilinear elliptic partial differential equation 
\begin{align}\label{eq:qrPDE}
-\diver \bigl( \langle G_f^{-1} \nabla u, \nabla u \rangle^{(n-2)/2} G_f^{-1} \nabla u \bigr) = 0 \, ,
\end{align}
where
\begin{displaymath}
G_f^{-1}(x) = \left\{ \begin{array}{ll}
\frac{\cof Df(x)^T \cof Df(x)}{\det Df(x)^{2(n-1)/n}}, & \textrm{if $J_f(x)>0$}\\
\I, & \textrm{otherwise,}
\end{array} \right.
\end{displaymath}
stands for the \emph{inverse dilatation tensor} and $\cof Df(x)$ denotes the \emph{cofactor matrix} of the differential matrix. This way one eventually obtains that
$$\mathcal{H}^1(f^{-1}(y)) = 0 \quad \text{for every } y \in \R^n \, ,$$
which implies that every quasiregular mapping is \emph{light} in the sense that the preimage of every point is totally disconnected. After this Reshetnyak's result follows by showing that quasiregular mappings are sense-preserving and by obtaining that sense-preserving and light maps between connected oriented manifolds are discrete and open, see e.g. \cite{BojarskiIwaniec1983, Heinonen2002,IwaniecMartin, ManfrediVillamor1998} for further details. 

Reshetnyak's theorem and its techniques have been studied further by several authors, see e.g. \cite{GoldsteinVodopyanov1976,HeinonenKoskela1993,HenclMaly2002,IwaniecSverak1993,ManfrediVillamor1998,ManfrediVillamor1995,OnninenZhong, Rajala2011}. In this process the developement of \emph{mappings of finite distortion} \cite{HenclKoskelaBook, IwaniecMartin, MRSYBook2009} and the study of their connection to the non-linear elasticity theory of Ball, Antman, and Ciarlet \cite{Ball2,Antman1976,Ci} have been main driving forces. In this context, the generalizations of Reshetnyak's theorem have been applied to investigate \emph{impenetrability of matter} of deformations in non-linear elasticity theory. However, usually these techniques have only been used to obtain discreteness and openness of deformations instead of recovering the actual local invertibility. 

In this article we study the local invertibility of quasiregular mappings by utilizing the earlier studies of Onninen and Zhong \cite{OnninenZhong} on Reshetnyak's theorem in order to study \emph{Martio's conjecture} which states that every non-constant quasiregular mapping
$$f : U \to f(U) \subset \R^n \quad (\text{$U \subset \R^n$ open set with $n \ge 3$})$$
with an \emph{inner dilatation}
$$K_I(f) \colonequals \esssup_{x \in U} \frac{\norm{\cof Df(x)}^n}{\det Df(x)^{n-1}}$$
less than two is a local homeomorphism. This long-standing unconfirmed conjecture was originally stated by Martio, Rickman, and Väisälä in \cite{MRV-71} and it was motivated by the preliminary work of Martio \cite{Martio1970}, where the conjecture was confirmed when the \emph{branch set}
$$\mathcal{B}_f \colonequals \{ x \in U : f \text{ is not a local homeomorphism at } x \}$$
of a quasiregular map contains a rectifiable curve. From a technical point of view it is usually more natural to study the following generalization of the conjecture from \cite{Tengvall}: 
\begin{conjecture}[Strong Martio's conjecture]
The inner dilatation of a non-constant quasiregular mapping
$$f : U \to f(U) \subset \R^n \quad (\text{$U \subset \R^n$ open set with $n \ge 3$})$$
satisfies
$$\inf_{x \in \mathcal{B}_f} i(x,f) \le K_I(f) \, ,$$ 
where and in what follows $i(x,f)$ 
stands for the local topological index of a point $x \in U$ under the mapping $f$, see \cite[Chapter~I]{Rickman-book}.
\end{conjecture}
We point out for the reader that the standard quasiregular \emph{$m$-to-1 winding mapping}
$$(r,\theta, z) \stackrel{w_m}{\mapsto} (r,m\theta,z) \quad \text{($z \in \R^{n-2}$)} \, ,$$
written here in the cylindrical coordinates, is an extermal for the conjecture. In addition, the holomorphic function
$$f : \mathbb{C} \to \mathbb{C}, \quad f(z) = z^m$$ 
shows the conjecture to fail in the planar case. In \cite{KLT3} Kauranen, Luisto, and Tengvall verified the conjecture for \emph{mappings of bounded length distortion}, also known as \emph{BLD-mappings}. This class consists of those quasiregular Lipschitz mappings
$$f : U \to f(U) \subset \R^n \quad \text{($U \subset \R^n$ open set with $n \ge 2$)}$$
for which we have
\begin{align}\label{eq:BoundedJacobian}
\det Df(x) > c \quad \text{a.e.}
\end{align}
for some constant $c>0$, see \cite{MartioVaisala}. In \cite{Tengvall} Tengvall relaxed the boundedness condition \eqref{eq:BoundedJacobian} even further by proving the conjecture under certain integrability condition on the reciprocal of the Jacobian determinant. He also offered several alternative proofs for the conjecture in the BLD-class. However, none of the above-mentioned proofs from \cite{KLT3,Tengvall} is self-contained as each one of them heavily relies on the following well-known local modulus of continuity estimate
\begin{align}\label{eq:ModuusOfContinuity}
\abs{f(x)-f(y)} \le C \abs{x-y}^{\bigl( \frac{i(x,f)}{K_I(f)}\bigr)^{\frac{1}{n-1}}} \quad \text{for all } y \in B(x,r)
\end{align} 
by Martio \cite{Martio1970} which can be also found from \cite[Theorem~III.4.7]{Rickman-book}. The proof of this estimate requires several layers of preliminary results which makes it technical and rather lengthy. In this article we prove the strong Martio's conjecture for BLD-mappings without any use of the estimate \eqref{eq:ModuusOfContinuity} by providing a rather short and self-contained proof for the following result from \cite{KLT3} which is valid also in the planar case:

\begin{theorem}[Kauranen, Luisto, and Tengvall, 2021]\label{thm:main}
	Every non-constant BLD-mapping
	$$f : U \to f(U) \subset \R^n \quad (\text{$U \subset \R^n$ open set with $n \ge 2$})$$
	satisfies $i(x,f) \le K_I(f)$ for every $x \in U$. 
\end{theorem}

Finally we highlight the connection of \emph{the generalized Liouville's theorem} to our studies on local invertibility of quasiregular mappings. This rigidity result states that non-planar, non-constant quasiregular mappings with
\begin{align}\label{eq:DilatationOne}
K_I(f) = 1
\end{align}
are restrictions of Möbius transformations. The result generalizes the well-known Liouville's theorem \cite{Capelli1886, Liouville1850,Hartman1947,Hartman1958} on rigidity of non-planar conformal diffeomorphisms. The original proofs of Gehring \cite{Gehring1962} and Reshetnyak \cite{ReshetnyakLiouville1967}  for the result are based on the study of regularity properties of the solutions to the non-linear $n$-harmonic equation
\begin{align}\label{eq:nHarmonic}
-\diver \bigl( \abs{\nabla u(x)}^{n-2} \nabla u(x) \bigr) = 0 \, .
\end{align}
This equation can be obtained from \eqref{eq:qrPDE} when 
$$G_f^{-1}(x) = \id \quad \text{a.e.}$$
and in the planar case it reduces to the usual Laplace equation. In the earlier-mentioned work \cite{MRV-71} of Martio, Rickman, and Väisälä (see also \cite{Goldstein1971}) generalized Liouville's theorem was applied with a compactness argument to obtain that non-planar, non-constant quasiregular mappings with an inner dilatation close to one are local homeomorphisms. Later a quantitative version of this result was obtained by Rajala \cite{Rajala-MartioResult}.

As by the generalized Liouville's theorem all non-planar, non-constant quasiregular mappings with the property \eqref{eq:DilatationOne} coincide with the identity map up to a conjugation by restrictions of Möbius transformations it is natural to ask whether similar kind of phenomenom occurs also for quasiregular mappings with
\begin{align}\label{InnerAndIndex}
K_I(f) = \inf_{x \in \mathcal{B}_f} i(x,f) \, .
\end{align}
In the light of current knowledge it seems that up to a conjugation by Möbius transformations the only non-planar quasiregular mapping with the property \eqref{InnerAndIndex} is the standard $m$-to-1 winding map. Therefore, we conjecture:

\begin{conjecture}[Rigidity conjecture]
Every non-planar quasiregular mapping with
\begin{align*}
K_I(f) = m_f, \quad \text{where} \quad m_f \colonequals \left\{ \begin{array}{ll}
1, & \textrm{if $\mathcal{B}_f = \emptyset$}\\[0.5em]
\inf_{x \in \mathcal{B}_f} i(x,f), & \textrm{if $\mathcal{B}_f \neq \emptyset$,}
\end{array} \right.
\end{align*}
equals to the standard $m_f$-to-1 winding mapping up to a conjugation by restrictions of Möbius transformations.
\end{conjecture}

\section*{\textbf{Acknowledgments}} The author wishes to thank Jani Onninen for several discussions on his work \cite{OnninenZhong}. He would also like to thank Katrin Fässler and Sebastiano Nicolussi Golo for their comments on this research after author's Jyväskylä Geometric Analysis Seminar presentation in November 2021. Part of the writing process of the article was done while the author was visiting Massey University, New Zealand Institute for Advanced Studies. He would like to thank the institute for its hospitality. The visit was funded by the Mobility Grant 2022 of Faculty of Mathematics and Science of University of Jyväskylä.

\section{Preliminaries}

In this section we recall some basic facts on mappings of bounded length distortion from \cite{MartioVaisala} and on discrete and open maps from \cite{Rickman-book}. If the reader is well-aware of the basic properties and the notation related to these mapping classes, reading this section is not necessary.

\subsection{Preliminary properties for BLD-maps} We recall that mappings of bounded length distortion form a subclass of quasiregular maps. Thus, it follows from Reshetnyak's theorem that these mappings are continuous, sense-preserving, discrete, and open. In addition, by the characterization \cite[Theorem~2.16]{MartioVaisala} of these mappings every BLD-map
$$f : U \to f(U) \subset \R^n \quad (\text{$U \subset \R^n$ open set with $n \ge 2$})$$
satisfies the following length distortion bounds
\begin{align}\label{eq:LengthDistortion}
\ell(\gamma)/L \le \ell(f \circ \gamma) \le L \ell(\gamma) 
\end{align}
for every path $\gamma$ in $U$ with some length distortion constant $L \ge 1$, where $\ell(\gamma)$ stands for the lenght of a path $\gamma$.

\subsection{Preliminary properies for discrete and open maps} As mappings of bounded lenght distortion form a subclass of discrete and open maps we may apply all the basic results on these maps in order to study BLD-maps further. In this section we recall the definitions and results on discrete and open maps from \cite[Chapter~I]{Rickman-book} that are later needed for the proof of Theorem~\ref{thm:main}. We start by recalling that an open, connected set $D \subset \subset U$ is called a \emph{normal domain} of a continuous, discrete, and open mapping
$$f : U \to f(U) \subset \R^n \quad (\text{$U \subset \R^n$ open set with $n \ge 2$})$$
if it satisfies
\begin{align*}
f(\partial D) = \partial f(D).
\end{align*}
If a normal domain $D$ satisfies
$$D \cap f^{-1}(f(x)) = \{ x\} \, ,$$
then it is called a \emph{normal neighborhood} of a point $x \in U$. In addition, for a given point $x \in U$ we denote 
$$U(x,f,r) \colonequals \text{``the $x$-component of the preimage $f^{-1}\bigl(B(f(x),r) \bigr)$''}.$$ 
With the notation introduced above we may recall the following standard lemma from \cite[Lemma~I.4.9]{Rickman-book} which is applied frequently throughout the article: 

\begin{lemma}\label{lemma:NormalDomain}
Let
$$f : U \to f(U) \subset \R^n \quad (\text{$U \subset \R^n$ open set with $n \ge 2$})$$
be a continuous, discrete, and open mapping. Then for every $x \in U$ there exists a radius $r_x > 0$ for which $U(x,f,r)$ is a normal neighborhood of $x$ such that
$$f\bigl(U(x,f,r) \bigr) = B\bigl(f(x),r\bigr) \quad \text{for every } 0 < r \le r_x.$$
Moreover, we have
$$\diam \bigl(U(x,f,r) \bigr) \to 0 \quad \text{as } r \to 0 \, .$$
\end{lemma}

\subsection{Path lifting} In order to study the compression of BLD-maps later in section~\ref{sec:Compression} we need the following path lifting lemma that follows directly from \cite[Proposition~I.4.10]{Rickman-book} and \cite[Corollary~II.3.4]{Rickman-book}:

\begin{lemma}\label{lemma:PathLifting}
Let
$$f : U \to f(U) \subset \R^n \quad (\text{$U \subset \R^n$ open set with $n \ge 2$})$$
be a continuous, discrete, and open mapping and let $D \subset \subset  U$ be a normal neighborhood of a point $x_0 \in U$ under the mapping $f$. In addition, denote
$$m \colonequals i(x_0,f) .$$	
Then for every path 
$$\beta :[a,b) \to f(D)$$ 
starting at $f(x_0)$ there exist paths
$$\alpha_j : [a,b) \to D \quad (\text{$j=1, \ldots, m$})$$
each of which starts at $x_0$ and such that the following conditions are satisfied: 
\begin{itemize}
	\item[(i)] $f \circ \alpha_j = \beta$ for each $j=1, \ldots, m$.
	\item[(ii)] $\card \{ j :\alpha_j(t) = x\} = i(x,f)$ for $x \in D \cap f^{-1}\bigl(\beta(t) \bigr)$. 
	\item[(iii)] For the traces of the above-mentioned paths we have the following relation
	$$\abs{\alpha_1} \cup \cdots \cup \abs{\alpha_m} = D \cap f^{-1}\bigl(\abs{\beta} \bigr).$$
\end{itemize}
\end{lemma}

\section{Compression of BLD-mappings}\label{sec:Compression}

In this section we shortly discuss the local compression of BLD-maps and show that these maps cannot compress materia too much together in the following sense:
\begin{lemma}\label{lemma:LowerBound}
	Every non-constant BLD-mapping
	$$f : U \to f(U) \subset \R^n \quad (\text{$U \subset \R^n$ open set with $n \ge 2$})$$
	satisfies 
	$$B\bigl(f(x),r/L \bigr) \subset f\bigl(B(x,r) \bigr) \quad \text{for some $L \ge 1$} \, ,$$ 
	whenever $r>0$ is sufficiently small.
\end{lemma}

\begin{proof}
Fix a point $x \in U$. By Lemma~\ref{lemma:NormalDomain} we may find a radius $r_x > 0$ such that the set $U(x,f,r)$ is a normal neighborhood of $x$ whenever $0 < r \le r_x$. Fix any radius $r > 0$ so small that
$$B(x,r) \subset U(x,r_x,f) \, ,$$
and denote by $\tilde{r} > 0$ the largest radius such that
$$B\bigl(f(x),\tilde{r} \bigr) \subset f\bigl(B(x,r) \bigr).$$
Then to conclude the proof it suffices to show that
\begin{align}\label{ProveThis}
\tilde{r} \ge r/L \quad \text{for some constant $L \ge 1$.}
\end{align}
For this purpose, we obtain that by the openness of $f$ we have
$$\partial f\bigl(B(x,r) \bigr) \subset f\bigl(\partial B(x,r)\bigr),$$
and therefore we may find a point $z \in \partial B(x,r)$ such that
$$f(z) \in \partial B\bigl(f(x),\tilde{r} \bigr).$$
Let us next consider the line-segment
$$I \colonequals [f(x),f(z)].$$
By applying Lemma~\ref{lemma:PathLifting} we may then find a path $\gamma$ in $U(x,f, \tilde{r})$ from $x$ to $z$ such that
$$f \circ \gamma = I \, .$$
However, then by \eqref{eq:LengthDistortion} we get
\begin{align*}
\tilde{r} = \ell(I) = \ell(f \circ \gamma) \ge \ell(\gamma)/L \ge r/L 
\end{align*}
for some constant $L \ge 1$. Thus, we have verified \eqref{ProveThis} and the claim follows.

\end{proof}

We point out that Lemma~\ref{lemma:LowerBound} can be actually obtained directly from \cite[Lemma~4.6]{MartioVaisala}. However, for the convenience of the reader and for the self-containedness of this article we have sketched a proof above as well. One should also notice that Lemma~\ref{lemma:LowerBound} is not valid for quasiregular maps in general. This can be demonstrated by the radially symmetric $K$-quasiconformal map 
\begin{align*}
f : \R^n \to \R^n, \quad f(x) = \left\{ \begin{array}{ll}
\frac{x}{\abs{x}} \abs{x}^{K^{\frac{1}{n-1}}}, & \textrm{if $x \neq 0$}\\[0.5em]
0, & \textrm{if $x=0$,}
\end{array} \right. 
\end{align*}
with a constant $K > 1$. In this context we note for the reader that obtaining Lemma~\ref{lemma:LowerBound} is the only step of our proof for Theorem~\ref{thm:main} where the BLD-assumption is needed as all the other steps are valid for general quasiregular mappings. 

\section{Improving the lemma of Onninen and Zhong}

In this section we provide the key lemma for our proof of Theorem~\ref{thm:main}. For this purpose we first recall the result of Onninen and Zhong from \cite{OnninenZhong} according to which for every
$$f \in C^{\infty}(U, \R^n) \quad \text{and} \quad \Psi \in C^1\bigl([0,\infty), [0,\infty) \bigr) \, ,$$
and for every compactly supported test-function $\eta \in C_c^{\infty}\bigl(U,[0,\infty) \bigr)$ we have
\begin{align*}
\bigg{\lvert} \, \int_{U} \eta^n [n \Psi\bigl(\abs{f}^2 \bigr) + 2\abs{f}^2 \Psi'\bigl(\abs{f}^2 \bigr) J_f] \, \bigg{\rvert} \le C \int_{U} \eta^{n-1}  \abs{\nabla \eta} \abs{f} \Psi\bigl(\abs{f}^2 \bigr) \norm{Df}^{n-1} \, , 
\end{align*}
where the constant $C>0$ depends on the dimension of the underlying open set $U \subset \R^n$. In what follows, we provide a sharp version of this estimate in Lemma~\ref{lemma:MainEstimate} below. The proof of this refinement is based on the following identity:

\begin{lemma}\label{lemma:BasicLinearAlgebra}
Let $U \subset \R^n$ be an open set with $n \ge 2$ and suppose that 
$$f : U \to f(U) \subset \R^n \quad \text{and} \quad \eta : U \to \R \quad $$
are differentiable at a point $x \in U$. Then
\begin{align*}
\sum_{i=1}^n f_i &J(x; f_1, \ldots, f_{i-1}, \eta, f_{i+1}, \ldots, f_n) = \nabla \eta(x)^T \cof Df(x) f(x) \, . 
\end{align*}
\end{lemma}

\begin{proof}
Let us denote
$$A_{i,j} \colonequals [\cof Df(x)]_{i,j} \, .$$
Then by direct computation we obtain that
\begin{align}\label{eq:CofactorIdetity}
\nabla \eta(x)^T \cof Df(x) f(x) &= \left[ \begin{array}{ccc}
\partial_1 \eta & \cdots & \partial_n \eta 
\end{array} \right] \left[\ \begin{array}{ccc}
A_{1,1} & \cdots & A_{1,n} \\
\vdots  & \ddots & \vdots \\
A_{n,1} & \cdots & A_{n,n}
\end{array} \right] \left[ \begin{array}{ccc}
f_1 \\
\vdots \\
f_n 
\end{array} \right] \nonumber\\
&= \biggl(\sum_{i=1}^n f_1 \partial_i \eta A_{i,1} \biggr) + \cdots + \biggl(\sum_{i=1}^n f_n \partial_i \eta A_{i,n} \biggr) \\
&= \sum_{i=1}^n \sum_{j=1}^n \partial_i \eta A_{i,j} f_j \, . \nonumber
\end{align}
On the other hand, for each $i=1, \ldots, n$ we have
\begin{align*}
 f_i &J(x; f_1, \ldots, f_{i-1}, \eta, f_{i+1}, \ldots, f_n) = \sum_{j=1}^n  \partial_j \eta A_{1,j} f_i \, ,
\end{align*}
and by summing over $i$'s we get
\begin{align}\label{eq:JacobianIdentity}
\sum_{i=1}^n f_i &J(x; f_1, \ldots, f_{i-1}, \eta, f_{i+1}, \ldots, f_n) = \sum_{i=1}^n \sum_{j=1}^n  \partial_j \eta A_{i,j} f_i  \, .
\end{align}
By combining \eqref{eq:CofactorIdetity}--\eqref{eq:JacobianIdentity} we have
\begin{align*}
\sum_{i=1}^n f_i &J(x; f_1, \ldots, f_{i-1}, \eta, f_{i+1}, \ldots, f_n) = \nabla \eta(x)^T \cof Df(x) f(x) \, , 
\end{align*}
which completes the proof.
\end{proof}

\begin{lemma}\label{lemma:MainEstimate}
Let $U \subset \R^n$ be an open set with $n \ge 2$ and suppose that
$$f \in C^{\infty}(U, \R^n) \quad \text{and} \quad \Psi \in C^1\bigl([0,\infty), [0,\infty) \bigr) \, .$$
Then for every test-function $\eta \in C_c^{\infty}(U, [0, \infty))$ we have
\begin{align*}
\bigg{\lvert} \, \int_{U} \eta \Bigl[n \Psi\bigl(\abs{f}^2 \bigr) + 2\abs{f}^2 \Psi'\bigl(\abs{f}^2 \bigr) J_f \Bigr] \, \bigg{\rvert} \le \int_{U}  \abs{\nabla \eta} \abs{f} \Psi\bigl(\abs{f}^2 \bigr) \norm{\cof Df} \, .
\end{align*}
\end{lemma}

\begin{proof}
The proof follows the same steps as the proofs of \cite[Lemma~2.1]{OnninenZhong} and \cite[Lemma~3.19]{HenclKoskelaBook}. Indeed, first by fixing $i \in \{ 1, \ldots, n\}$ we obtain by applying Stokes' theorem that	
\begin{align}\label{eq:stokes}
\int_{U} J(x; f_1, \ldots, f_{i-1}, \eta \Psi\bigl(\abs{f}^2 \bigr) f_i, f_{i+1}, \ldots, f_n ) \, dx = 0 \, .
\end{align}
On the other hand, the chain rule gives us
\begin{align*}
J(x; f_1, \ldots, f_{i-1}, &\Psi\bigl(\abs{f}^2 \bigr), f_{i+1}, \ldots, f_n) \\
&= \sum_{j=1}^n 2\Psi'\bigl(\abs{f}^2 \bigr) f_j J(x; f_1, \ldots, f_{i-1},f_j,f_{i+1}, \ldots, f_n) \\
&= 2\Psi'\bigl(\abs{f}^2 \bigr) f_i J(x; f_1, \ldots, f_n) \, .
\end{align*}
Thus, by the product rule we obtain
\begin{align*}
&J(x; f_1, \ldots, f_{i-1}, \eta \Psi\bigl(\abs{f}^2 \bigr) f_i, f_{i+1}, \ldots, f_n ) \\
&= \Psi\bigl(\abs{f}^2 \bigr) f_i J(x; f_1, \ldots, f_{i+1},\eta, f_{i+1}, \ldots, f_n) + 2\eta f_i^2 \Psi'\bigl(\abs{f}^2 \bigr) J_f + \eta \Psi\bigl(\abs{f}^2 \bigr) J_f \, .
\end{align*}
By combining this with \eqref{eq:stokes} we get
\begin{align*}
\int_{U} \eta \bigl[ \Psi\bigl(\abs{f}^2 \bigr) + &2\Psi'\bigl(\abs{f}^2 \bigr) f_i^2\bigr] J_f \\
&= -\int_{U}  \Psi\bigl(\abs{f}^2 \bigr) f_i J(x; f_1, \ldots, f_{i-1}, \eta, f_{i+1}, \ldots, f_n) \, ,
\end{align*}
and by summing over $i$'s and applying Lemma~\ref{lemma:BasicLinearAlgebra} we have
\begin{align*}
\biggl{\lvert} \int_{U} \eta \Bigl( n&\Psi\bigl(\abs{f}^2 \bigr) + 2\Psi'\bigl(\abs{f}^2 \bigr) \abs{f}^2\Bigr) J_f \biggr{\rvert}\\
&= \biggl{\lvert} \sum_{i=1}^n \int_{U} \Psi\bigl(\abs{f}^2 \bigr) f_i J(x; f_1, \ldots, f_{i-1}, \eta, f_{i+1}, \ldots, f_n) \biggr{\rvert} \\
&= \biggl{\lvert} \int_{U}  \Psi\bigl(\abs{f(x)}^2 \bigr) \nabla  \eta(x)^T \cof Df(x) f(x) \biggr{\rvert} \\
&= \biggl{\lvert} \int_{U \setminus \{ x : \abs{f(x)} = 0\}}  \Psi\bigl(\abs{f(x)}^2 \bigr) \nabla  \eta(x)^T \cof Df(x) f(x) \biggr{\rvert} \\
&\le \int_{U \setminus \{ x : \abs{f(x)} = 0\}} \Psi\bigl(\abs{f(x)}^2 \bigr) \abs{\nabla \eta(x)} \, \abs{f(x)} \,  \bigg{\lvert} \cof Df(x) \frac{f(x)}{\abs{f(x)}} \bigg{\rvert} \, dx \\
&\le \int_{U}  \Psi\bigl(\abs{f(x)}^2 \bigr) \abs{\nabla \eta(x)} \, \abs{f(x)} \,   \norm{\cof Df(x)}\, dx \, ,
\end{align*}
which ends the proof.
\end{proof}

\section{Proof of Theorem~\ref{thm:main}}

To prove Theorem~\ref{thm:main} we first assume without loss of generality that
$$f(0) = 0 \quad \text{ and } \quad i(0,f) = m.$$ 
By applying Lemma~\ref{lemma:NormalDomain} we find a radius $r_0>0$ such that $U(0,f,t)$ defines a normal neighborhood of the origin whenever $0 < t \le r_0$. Next, we fix a radius  
$$0 < R  < \min \{ 1, r_0\} \quad \text{so small that} \quad f\bigl(B(0,R) \bigr) \subset B(0,r_0) ,$$
and assume first that 
$$f \in C^{\infty}(U, \R^n).$$ 
Suppose now that $t \in (0,R)$ and let $\varepsilon >0$. By applying Lemma~\ref{lemma:MainEstimate} with the standard cut-off function 
\begin{align*}
\eta_{\varepsilon}(s) = \left\{ \begin{array}{ll}
1, & \textrm{if $0 \le s \le t-\varepsilon$}\\
\frac{t-s}{\varepsilon} \, ,& \textrm{if $t-\varepsilon < s < t$}\\
0, & \textrm{if $s \ge t$,}
\end{array} \right.
\end{align*}
we get for any function $\Psi \in C^1\bigl([0,\infty),[0,\infty) \bigr)$ the following estimate
\begin{align}\label{eq:IntegralDifferation}
\biggl{\lvert}\int_{B_t} \eta_{\varepsilon} \bigl[ n \Psi(\abs{f}^2) +2 \abs{f}^2 \Psi'(\abs{f}^2)\bigr] J_f \bigg{\rvert} &\le \int_{B_t} \abs{\nabla \eta_{\varepsilon}} \abs{f} \Psi(\abs{f}^2) \norm{\cof Df} \nonumber \\
&= \frac{1}{\varepsilon} \int_{B_t \setminus B_{t-\varepsilon}} \abs{f} \Psi(\abs{f}^2) \norm{\cof Df} \, ,
\end{align}
where and in what follows we denote
$$B_t \colonequals B(0,t) \quad \text{for a given radius } t > 0 \, .$$
By letting $\varepsilon \to 0$ we conclude from \eqref{eq:IntegralDifferation} that
\begin{align}\label{eq:BoundaryEstimate}
\biggl{\lvert}\int_{B_t} \bigl[ n \Psi(\abs{f}^2) +2 \abs{f}^2 \Psi'(\abs{f}^2)\bigr] J_f \bigg{\rvert} \le \int_{\partial B_t} \abs{f} \Psi(\abs{f}^2) \norm{\cof Df} \, d \mathcal{H}^{n-1},
\end{align}
see \cite[3.4.4]{EvansGariepyBook}. Next we observe that for the function
\begin{align*}
\Psi(s) = \frac{s^{-\frac{n}{2}}}{\log^{n-1}(1/s)}
\end{align*}
we have
\begin{align*}
\Psi'(s) = -\frac{n}{2\log^{n-1}(1/s) s^{\frac{n}{2} +1}} + \frac{n-1}{\log^n(1/s) s^{\frac{n}{2} +1}} \, .
\end{align*}
By approximating $\Psi$ with $C^1$-functions we conclude from \eqref{eq:BoundaryEstimate} the following estimate
\begin{align}\label{eq:BestEstimate}
2(n-1) \int_{B_t} \frac{J_f}{\abs{f}^n \log^n(1/\abs{f}^2)} &\le  \int_{\partial B_t} \frac{\norm{\cof Df}}{\abs{f}^{n-1} \log^{n-1}(1/\abs{f}^2)} \, d \mathcal{H}^{n-1} \, .
\end{align}
Now, the standard approximation argument shows that for almost every radius $t \in (0,R)$ the estimate  \eqref{eq:BestEstimate} is valid also for continuous Sobolev mappings 
$$f \in W_{\loc}^{1,n}(U, \R^n) \cap C(U,\R^n) \, .$$
Therefore, under the assumptions of the theorem we obtain by applying the estimate \eqref{eq:BestEstimate} together with the inner dilatation inequality
$$\norm{\cof Df(x)}^n \le K_I(f) J_f(x)^{n-1} \colonequals K_I J_f(x)^{n-1} \quad \text{a.e.} \, ,$$
and Hölder's inequality that for almost every $t \in (0,R)$ one has
\begin{align}\label{eq:BasicEstimation}
2(n-1) \int_{B_t} \frac{J_f}{\abs{f}^n \log^n(1/\abs{f}^2)} &\le  \int_{\partial B_t} \frac{\norm{\cof Df}}{\abs{f}^{n-1} \log^{n-1}(1/\abs{f}^2)}  \nonumber \\ 
&\le K_I^{\frac{1}{n}} \int_{\partial B_t} \frac{J_f^{\frac{n-1}{n}}}{\abs{f}^{n-1}\log^{n-1}(1/\abs{f}^2)} \nonumber \\
&\le K_I^{\frac{1}{n}} \abs{\partial B_t}^{\frac{1}{n}} \biggl( \int_{\partial B_t} \frac{J_f}{\abs{f}^n \log^{n}(1/\abs{f}^2)} \biggr)^{\frac{n-1}{n}} \\
&= K_I^{\frac{1}{n}}  \omega_{n-1}^{\frac{1}{n}} t^{\frac{n-1}{n}} \biggl( \int_{\partial B_t} \frac{J_f}{\abs{f}^n \log^{n}(1/\abs{f}^2)} \biggr)^{\frac{n-1}{n}} \, , \nonumber
\end{align}
where $\omega_{n-1}$ denotes the surface measure of the unit sphere and $\abs{\partial B_t}$ denotes the surface measure of the sphere $\partial B_t$. By considering the increasing, non-negative function
\begin{align*}
\varphi(t) = \int_{B_t} \frac{J_f}{\abs{f}^n \log^n(1/\abs{f}^2)}
\end{align*}
we may rewrite the estimate \eqref{eq:BasicEstimation} as follows
\begin{align*}
\varphi(t) \le C_n t^{\frac{n-1}{n}} \bigl[ \varphi'(t) \bigr]^{\frac{n-1}{n}} \, , \quad \text{where} \quad C_n \colonequals \frac{K_I^{\frac{1}{n}}\omega_{n-1}^{\frac{1}{n}}}{2(n-1)} \, .
\end{align*}
Therefore, we have
\begin{align*}
\frac{1}{C_n^{\frac{n}{n-1}} t} \le \frac{\varphi'(t)}{\varphi(t)^{\frac{n}{n-1}}} = -(n-1) \frac{d}{dt} \, \biggl[ \varphi(t)^{-\frac{1}{n-1}} \biggr] \, ,
\end{align*}
and by integrating both sides over an interval $(r,R)$ we get
\begin{align*}
\frac{1}{C_n^{\frac{n}{n-1}}} \log(R/r) &\le -(n-1) \biggl[ \varphi(R)^{-\frac{1}{n-1}} - \varphi(r)^{-\frac{1}{n-1}}  \biggr] \le (n-1) \varphi(r)^{-\frac{1}{n-1}} \, .
\end{align*}
This gives us
\begin{align}\label{eq:UpperMainEstimate}
\varphi(r) \le \frac{C_n^n (n-1)^{n-1}}{\log^{n-1}(R/r)} \, .
\end{align}
Next, we recall that by Lemma~\ref{lemma:LowerBound} we have
\begin{align}\label{eq:ImportantInclusion}
B_{r/L} \subset f(B_r) \quad \text{for all sufficiently small radii $r>0$}.
\end{align}
In addition, for a given set $A \subset \R^n$ we recall the standard notation
$$N(y,f,A) \colonequals \card f^{-1}(y) \cap A \, .$$
Then by applying the change of variables formula \cite[Theorem~A.35]{HenclKoskelaBook} and the inclusion \eqref{eq:ImportantInclusion} we obtain the following lower estimate
\begin{align}\label{eq:LowerMainEstimate1}
\varphi(r) = \int_{B_r} \frac{J_f}{\abs{f}^n \log^n(1/\abs{f}^2)} &= \int_{f(B_r)} \frac{N(y,f,B_r)}{\abs{y}^n \log^{n}(1/\abs{y}^2)} \, dy \nonumber \\
&\ge \int_{B_{r/L}} \frac{N(y,f,U_{r/L})}{\abs{y}^n \log^{n}(1/\abs{y}^2)} \, dy \\
&= i(0,f) \int_{B_{r/L}} \frac{dy}{\abs{y}^n \log^{n}(1/\abs{y}^2)} \nonumber
\end{align}
for all sufficiently small $r>0$, where the last equality follows from the fact that 
$$U_{r/L} \colonequals U(0,f,r/L)$$ 
is a normal neighborhood of the origin and from \cite[Proposition~I.4.10]{Rickman-book}. By a direct calculation we get
\begin{align}\label{eq:LowerMainEstimate2}
\int_{B_{r/L}} \frac{dz}{\abs{y}^n\log^{n}(1/\abs{y}^2)} &= \int_0^{r/L} \int_{\partial B_t} \frac{1}{t^n \log^n(1/t^2)} \nonumber \\
&= \frac{\omega_{n-1}}{2^n} \int_0^{r/L} \frac{dt}{t\log^n(1/t)} \\
&= -\frac{\omega_{n-1}}{2^n} \int_{\infty}^{\log(L/r)} \frac{ds}{s^n} \nonumber \\
&= \frac{\omega_{n-1}}{(n-1)2^n \log^{n-1}\bigl(\frac{L}{r} \bigr)}  \nonumber
\end{align}
for all sufficiently small radii $r>0$. Thus, by combining the estimates \eqref{eq:UpperMainEstimate} and \eqref{eq:LowerMainEstimate1}--\eqref{eq:LowerMainEstimate2} we get
\begin{align*}
\frac{i(0,f)\omega_{n-1}}{(n-1)2^n \log^{n-1}\bigl(\frac{L}{r} \bigr)} \le \varphi(r) &\le \frac{C_n^n(n-1)^{n-1}}{\log^{n-1}(R/r)} = \frac{K_I \omega_{n-1}(n-1)^{n-1}}{2^n (n-1)^n\log^{n-1}(R/r)} \, .
\end{align*}
By simplifying both sides we finally obtain
\begin{align*}
i(0,f) \le K_I \biggl( \frac{\log(L/r)}{\log(R/r)} \biggr)^{n-1} \quad \text{for all small } r>0 \, .
\end{align*}
Thus, by letting $r \to 0$ we observe that $i(0,f) \le K_I$ which concludes the proof.

\section{Final remarks}

We close the article with further observations on rigidity and Martio's conjectures. We start by pointing out that in the literature several constructions on quasiregular mappings \cite{BonkHeinonen-Smooth, GehringVaisala1973,KaufmanTysonWu2005}, on mappings of bounded length distortion \cite{MartioVaisala}, and on mappings of finite distortion \cite{GuoHenclTengvall2020} have the conjugated form
\begin{align}\label{eq:ConjugatedFrom}
h \circ w_m \circ g \, ,
\end{align}
where $w_m$ stands for the standard $m$-to-1 winding map, and both $h$ and $g$ are some suitable homeomorphisms. In addition to this, by the results of Church and Hemmingsen \cite[Theorem~4.1]{ChurchHemmingsen1960}, Martio, Rickman, and Väisälä \cite[Lemma~3.20]{MRV-71}, and Luisto and Prywes \cite[Theorem~1.1]{LuistoPrywes2021} quasiregular maps with reasonably regular branch sets or branch set images are indeed topologically equivalent to the standard winding map. Therefore, studying this type of maps is not useful only for the sake of the rigidity and Martio's conjectures, but also for its own right. On the other hand, from the point of view of these conjectures the following question stands out:

\begin{question}\label{question1}
Does the standard $m$-to-1 winding map minimize the inner dilatation for the non-planar quasiregular maps of the conjugated form \eqref{eq:ConjugatedFrom} in the following sense
\begin{align*}
K_I(w_m) = \min \biggl{\{} K_I(h \circ w_m \circ g) : \text{$g$ and $h$ quasiconformal} \biggr{\}} \, ?
\end{align*}
\end{question}

By answering the question above one would solve Martio's conjecture for the conjugated quasiregular mappings. As one may notice, Theorem~\ref{thm:main} already anwers the question in the class of BLD-maps. Therefore, for BLD-maps a natural next step is to study the uniqueness of the winding map by investigating the accuracy of the rigidity conjecture.

We also point out that by the constructions in \cite{BonkHeinonen-Smooth,KaufmanTysonWu2005} there exists even reasonably smooth quasiregular maps with branching. Furthermore, these examples can be even written as \eqref{eq:ConjugatedFrom} for some quasiconformal maps $h$ and $g$. Therefore, by current knowledge studying Martio's conjecture in the class of continuously differentiable quasiregular maps makes perfect sense as well as investigating the following question:

\begin{question}
Let us consider a continuously differentiable non-planar quasiregular map
$$f : \R^n \to \R^n, \quad f = h \circ w_m \circ g \, ,$$
where $g$ and $h$ are quasiconformal maps. Does it follow that $K_I(f) \ge m$?
\end{question}

\newcommand{\etalchar}[1]{$^{#1}$}
\def\cprime{$'$}


\begin{thebibliography}{GMRV00}
	
\bibitem[Ant76]{Antman1976}
S.~S.~Antman.
\newblock {\em Fundamental mathematical problems in the theory of nonlinear
elasticity}, Numerical solution of partial differential equations, III, Academic Press, New York, 35--54, 1976.

	
\bibitem[AIM09]{Astala-Iwaniec-Martin}
K.~Astala, T.~Iwaniec, and G.~Martin.
\newblock {\em Elliptic partial differential equations and quasiconformal
mappings in the plane}, volume~48 of {\em Princeton Mathematical Series}.
\newblock Princeton University Press, Princeton, NJ, 2009.

\bibitem[Ball81]{Ball2}
J.~M.~Ball 
\newblock {Global invertibility of {S}obolev functions and the interpenetration of matter}. 
\newblock {\em Proc. Roy. Soc. Edinburgh Sect. A}, 88(3--4):315--328, 1981.

\bibitem[BI83]{BojarskiIwaniec1983}
B.~Bojarski and T. Iwaniec.
\newblock {Analytical foundations of the theory of quasiconformal
	mappings in {${\bf R}^{n}$}}
\newblock {\em Ann. Acad. Sci. Fenn. Ser. A I Math.}, 8(2): 257--324, 1983.

\bibitem[BH04]{BonkHeinonen-Smooth}
M.~Bonk and J.~Heinonen.
\newblock Smooth quasiregular mappings with branching.
\newblock {\em Publ. Math. Inst. Hautes \'{E}tudes Sci.}, (100):153--170, 2004.

\bibitem[Cap86]{Capelli1886}
A.~Capelli.
\newblock  Sulla limitata possibility di trasformazioni conformi nello spazio.
\newblock {\em Annali di Matematica}, 14:227--237, 1886.


\bibitem[\v{C}e64]{Chernavski1964}
A.~V.~\v{C}ernavski\u{\i}.
\newblock Finite-to-one open mappings of manifolds.
\newblock {\em Mat. Sb. (N.S.)}, 65 (107):357--369, 1964.


\bibitem[\v{C}e65]{Chernavski1965}
A.~V.~\v{C}ernavski\u{\i}.
\newblock Addendum to the paper ``Finite-to-one open mappings of
manifolds''.
\newblock {\em Mat. Sb. (N.S.)}, 66 (108):471--472, 1965.


\bibitem[CH60]{ChurchHemmingsen1960}
P.~T.~Church and E.~Hemmingsen.
\newblock Light open maps on {$n$}-manifolds.
\newblock {\em Duke Math. J.}, 27:527--536, 1960.

\bibitem[CH61]{ChurchHemmingsen1961}
P.~T.~Church and E.~Hemmingsen.
\newblock Light open maps on {$n$}-manifolds. {II}.
\newblock {\em Duke Math. J.}, 28:607--623, 1961.

\bibitem[CH63]{ChurchHemmingsen1963}
P.~T.~Church and E.~Hemmingsen.
\newblock Light open maps on {$n$}-manifolds. {III}.
\newblock {\em Duke Math. J.}, 30:379--389, 1963.

\bibitem[Cia88]{Ci}
P.~G.~Ciarlet.
\newblock {\em Mathematical elasticity, Volume}~\textbf{I}: {\em Three-dimensional elasticity},  Studies in Mathematics and its Applications, 20. North-Holland publishing Co., Amsterdam, 1988.

\bibitem[EG92]{EvansGariepyBook}
L.~C.~Evans and R.~F.~Gariepy.
\newblock {\em Measure theory and fine properties of functions},
\newblock {Studies in Advanced Mathematics. CRC Press, Boca Raton, FL, 1992.}

\bibitem[Geh62]{Gehring1962}
F.~W. Gehring.
\newblock Rings and quasiconformal mappings in space.
\newblock {\em Trans. Amer. Math. Soc.}, 103:353--393, 1962.

\bibitem[GV73]{GehringVaisala1973}
F.~W. Gehring and J.~V\"{a}is\"{a}l\"{a}.
\newblock Hausdorff dimension and quasiconformal mappings.
\newblock {\em J. London Math. Soc. (2)}, 6:504--512, 1973.

\bibitem[Gol71]{Goldstein1971}
V.~M.~Gol'd\v{s}te\u{\i}n.
\newblock The behavior of mappings with bounded distortion when the distortion
coefficient is close to one.
\newblock {\em Sibirsk. Mat. \v{Z}.}, 12:1250--1258, 1971.

\bibitem[GoVy76]{GoldsteinVodopyanov1976}
V.~M.~Gol'd\v{s}te\u{\i}n and S.~K.~Vodop’yanov.
\newblock Quasiconformal mappings, and spaces of functions with first generalized derivatives.
\newblock {\em Sibirsk. Mat. \v{Z}.}, 17(3):515--531, 715, 1976.

\bibitem[GHT20]{GuoHenclTengvall2020}
C.-Y.~Guo, S.~Hencl, and V.~Tengvall.
\newblock Mappings of finite distortion: size of the branch set.
\newblock {\em Adv. Calc. Var.},  13(4):325--360, 2020.


\bibitem[Har47]{Hartman1947}
P.~Hartman.
\newblock Systems of total differential equations and Liouville’s theorem on conformal mappings.
\newblock {\em Amer. J. Math.}, 69:327--332, 1947.

\bibitem[Har58]{Hartman1958}
P.~Hartman.
\newblock On isometries and on a theorem of Liouville.
\newblock {\em Math. Z.}, 69:202--210, 1958.


\bibitem[Hei02]{Heinonen2002}
J.~Heinonen.
\newblock The branch set of a quasiregular mapping.
\newblock {\em Proceedings of the {I}nternational {C}ongress of {M}athematicians, {V}ol. {II} ({B}eijing, 2002)}, pages 691--700. Higher Ed.
Press, Beijing, 2002.

\bibitem[HK93]{HeinonenKoskela1993}
J.~Heinonen and P.~Koskela.
\newblock  Sobolev mappings with integrable dilatations.
\newblock {\em Arch. Rational Mech. Anal.}, 125(1):81--97, 1993.

\bibitem[HeKo14]{HenclKoskelaBook}
S.~Hencl and P.~Koskela. 
\newblock {\em Lectures on mappings of finite distortion}, volume~2096 of Lecture Notes in Mathematics. Springer, Cham, 2014.

\bibitem[HM02]{HenclMaly2002}
S.~Hencl and J.~Mal\'{y}. 
\newblock Mappings of finite distortion: Hausdorff measure of zero sets. \newblock {\em Math. Ann.}, 324(3):451--464, 2002.

\bibitem[IM01]{IwaniecMartin}
T.~Iwaniec and G.~Martin.
\newblock {\em Geometric function theory and non-linear analysis}.
\newblock Oxford Mathematical Monographs. The Clarendon Press, Oxford
University Press, New York, 2001.

\bibitem[IS93]{IwaniecSverak1993}
T.~Iwaniec and V.~\v{S}ver\'{a}k.
\newblock On mappings with integrable dilatation.
\newblock {\em Proc. Amer. Math. Soc.}, 118(1):181--188, 1993.

\bibitem[KTW05]{KaufmanTysonWu2005}
R.~Kaufman, J.~T. Tyson, and J.-M. Wu.
\newblock Smooth quasiregular maps with branching in {${\bf R}^n$}.
\newblock {\em Publ. Math. Inst. Hautes \'{E}tudes Sci.}, (101):209--241, 2005.


\bibitem[KLT21]{KLT3}
A.~Kauranen, R.~Luisto, and V.~Tengvall.
\newblock On {BLD}-mappings with small distortion.
\newblock {\em  Complex Anal. Synerg.}, 7(1): 1--4 (Paper No. 5), 2021.

\bibitem[Lio50]{Liouville1850}
J.~Liouville.
\newblock Théor\`{e}me sur l'équation $dx^2+dy^2+dz^2=\lambda (d\alpha^2+d\beta^2+d\gamma^2)$.
\newblock J. {\em Math. Pures Appl.}, 1(15):103, 1850.

\bibitem[LP21]{LuistoPrywes2021}
R.~Luisto and E.~Prywes.
\newblock Open and discrete maps with piecewise linear branch set images are piecewise linear maps.
\newblock {\em J. Lond. Math. Soc. (2)}, 103(3):1186--1207, 2021.

\bibitem[MV95]{ManfrediVillamor1995}
J.~J.~Manfredi and E.~Villamor.
\newblock  Mappings with integrable dilatation in higher dimensions.
\newblock {\em Bull. Amer. Math. Soc. (N.S.)},  32(2):235--240, 1995. 


\bibitem[MV98]{ManfrediVillamor1998}
J.~J.~Manfredi and E.~Villamor.
\newblock An extension of Reshetnyak’s theorem.
\newblock {\em Indiana Univ. Math. J.}, 47(3):1131--1145, 1998. 

\bibitem[Mar70]{Martio1970}
O.~Martio.
\newblock A capacity inequality for quasiregular mappings.
\newblock {\em Ann. Acad. Sci. Fenn. Ser. A. I.}, (474):18, 1970.

\bibitem[MRV71]{MRV-71}
O.~Martio, S.~Rickman, and J.~V\"{a}is\"{a}l\"{a}.
\newblock Topological and metric properties of quasiregular mappings.
\newblock {\em Ann. Acad. Sci. Fenn. Ser. A I}, (488):1--31, 1971.

\bibitem[MRSY09]{MRSYBook2009}
O.~Martio, V.~Ryazanov, U.~Srebro, and E.~Yakubov.
\newblock {\em Moduli in modern mapping theory}, Springer Monographs in Mathematics. Springer, New York, 2009.

\bibitem[MV88]{MartioVaisala}
O.~Martio and J.~V\"{a}is\"{a}l\"{a}.
\newblock Elliptic equations and maps of bounded length distortion.
\newblock {\em Math. Ann.}, 282(3):423--443, 1988.


\bibitem[OZ08]{OnninenZhong}
J. Onninen and X. Zhong.
\newblock  Mappings of finite distortion: a new proof for discreteness and openness.
\newblock {\em Proc. Roy. Soc. Edinburgh Sect. A}, 138(5):1097--1102, 2008.


\bibitem[Raj05]{Rajala-MartioResult}
K.~Rajala.
\newblock The local homeomorphism property of spatial quasiregular mappings
with distortion close to one.
\newblock {\em Geom. Funct. Anal.}, 15(5):1100--1127, 2005.

\bibitem[Raj11]{Rajala2011}
K.~Rajala.
\newblock Reshetnyak's theorem and the inner distortion.
\newblock {\em Pure Appl. Math. Q.}, 7(2):411--424, 2011.


\bibitem[Res67]{ReshetnyakLiouville1967}
Yu.~G. Reshetnyak.
\newblock Liouville's conformal mapping theorem under minimal regularity
hypotheses.
\newblock {\em Sibirsk. Mat. \v{Z}.}, 8:835--840, 1967.


\bibitem[Res89]{Reshetnyak89}
Yu.~G. Reshetnyak.
\newblock {\em Space mappings with bounded distortion}, volume~73 of {\em
	Translations of Mathematical Monographs}.
\newblock American Mathematical Society, Providence, RI, 1989.



\bibitem[Ric93]{Rickman-book}
S.~Rickman.
\newblock {\em Quasiregular mappings}, volume~26.
\newblock Springer-Verlag, Berlin, 1993.

\bibitem[Ten]{Tengvall}
V.~Tengvall.
\newblock Remarks on Martio’s conjecture.
\newblock {\em  To appear in Math. Scand.}



\bibitem[Vuo88]{VuorinenBook}
M.~Vuorinen.
\newblock {\em Conformal geometry and quasiregular mappings}, volume 1319 of
{\em Lecture Notes in Mathematics}.
\newblock Springer-Verlag, Berlin, 1988.

\bibitem[Vä71]{VaisalaBook}
J.~Väis\"{a}l\"{a}.
\newblock {\em Lectures on {$n$}-dimensional quasiconformal mappings}.
\newblock Lecture Notes in Mathematics, Vol. 229. Springer-Verlag, Berlin,
1971.




	
\end{thebibliography}
\end{document}